\documentclass[11pt]{amsart}
\usepackage{amsmath, amssymb}
\usepackage{amsfonts}
\usepackage{mathrsfs}
\usepackage[arrow,matrix,curve,cmtip,ps]{xy}

\usepackage{amsthm}

\allowdisplaybreaks

\newtheorem{theorem}{Theorem}[section]
\newtheorem{lemma}[theorem]{Lemma}
\newtheorem{proposition}[theorem]{Proposition}
\newtheorem{corollary}[theorem]{Corollary}

\newtheorem*{theorem*}{Theorem}
\theoremstyle{remark}
\newtheorem{remark}[theorem]{Remark}
\newtheorem{definition}[theorem]{Definition}

\numberwithin{equation}{section}

\newcommand{\R}{\mathbb{R}}
\newcommand{\C}{\mathbb{C}}

\newcommand{\Hi}{\mathcal{H}}
\newcommand{\im}{\operatorname{im }}

\newcommand{\gae}{\lower 2pt \hbox{$\, \buildrel {\scriptstyle >}\over {\scriptstyle
\sim}\,$}}

\newcommand{\lae}{\lower 2pt \hbox{$\, \buildrel {\scriptstyle <}\over {\scriptstyle
\sim}\,$}}

\newcommand{\MU}[1]{
\setbox0\hbox{$#1$}
\setbox1\hbox{$W$}
\ifdim\wd0>\wd1 #1^{\sim} \else \widetilde{#1} \fi
}

\begin{document}
\title[Continuity of ring $*$-homomorphisms between $C^*$-algebras]{Continuity of ring $\boldsymbol{*}$-homomorphisms between $\boldsymbol{C^*}$-algebras}

\author{Mark Tomforde 
}

\address{Department of Mathematics \\ University of Houston \\ Houston, TX 77204-3008 \\USA}
\email{tomforde@math.uh.edu}


\date{May 5, 2009}



\begin{abstract}
The purpose of this note is to prove that if $A$ and $B$ are unital $C^*$-algebras and $\phi : A \to B$ is a unital $*$-preserving ring homomorphism, then $\phi$ is contractive; i.e., $\| \phi (a) \| \leq \| a \|$ for all $a \in A$.   (Note that we do not assume $\phi$ is linear.)  We use this result to deduce a number of corollaries as well as characterize the form of such unital $*$-preserving ring homomorphisms. 
\end{abstract}

\maketitle

\section{Introduction}

In recent years there has been a great deal of interaction between ring theory and the study of $C^*$-algebras in functional analysis.  Consequently, one is often led to ask whether well-known results for $C^*$-algebras are really consequences of their ring structure. 

One of the first facts a student of $C^*$-algebras learns is that a $*$-ho\-mo\-morph\-ism between $C^*$-algebras is contractive, and hence also continuous.  Specifically, if $A$ and $B$ are $C^*$-algebras and $\phi : A \to B$ is a $*$-homomorphism (i.e., $\phi$ is linear, multiplicative, and preserves the $*$-operations) then $\| \phi (a) \| \leq \| a \|$ for all $a \in A$.  The purpose of this note is to show that the same conclusion holds when the map is a $*$-preserving ring homomorphism that is not necessarily linear.  We develop a new approach to prove this fact in Theorem~\ref{main-thm} and deduce a number of corollaries.  Furthermore, we use these results to characterize such homomorphisms in Proposition~\ref{characterize-prop}, where we show each is always a direct sum of a linear $*$-preserving homomorphism and a conjugate-linear $*$-preserving homomorphism.

We hope the results presented here will be of interest to algebraists as well as functional analysts.  Consequently, we will do our best to keep the exposition accessible to readers outside of functional analysis and to provide the necessary background on $C^*$-algebras and their relevant properties.

\section{Background}

Recall that a $C^*$-algebra is defined to be a Banach algebra $A$ over the complex numbers (i.e., $A$ is an associative complex algebra with a submultiplicative norm $\| \cdot \|$ for which $A$ is complete) together with a $*$-operation (i.e., an involution that is conjugate-linear and anti-multiplicative) satisfying the $C^*$-identity: $\| x^*x \| = \| x \|^2$ for all $x \in A$.  If $A$ is a $C^*$-algebra, the Gelfand-Naimark Theorem states that there is a Hilbert space $\Hi$ and an isometric $*$-preserving representation $\pi : A \to \mathcal{B} ( \Hi)$.  Thus $C^*$-algebras give an abstract characterization of closed $*$-subalgebras of bounded operators on Hilbert space. 

One of the fundamental results in the study of $C^*$-algebras is that if $\phi: A \to B$ is a map between $C^*$-algebras that is linear, multiplicative, and preserves the $*$-operations, then $\| \phi (a) \| \leq \| a \|$ for all $a \in A$.  The standard way that this is proven (see \cite[Theorem~2.1.7]{Mur}, for example) is as follows:  One first realizes that by considering unitizations of $C^*$-algebras, it suffices to restrict to the case that $A$ and $B$ are unital $C^*$-algebras and $\phi$ is a unital homomorphism.  Next one shows that if $a \in A$ is selfadjoint (i.e., $a^* = a$), then $\| a \|$ is equal to the spectral radius $$r(a) = \sup \{ | \lambda | : \lambda \in \C \text{ and } a-\lambda 1_A \text{ is not invertible in $A$} \}.$$  Then, using the properties of $\phi$, one can see that $a - \lambda 1_A$ is invertible in $A$ implies $\phi(a - \lambda 1_A)$ is invertible in $B$.  Since $\phi$ is linear and unital, $\phi (a - \lambda 1_A) = \phi(a) - \lambda 1_B$, and we have 
\begin{align}
\{ | \lambda | : \lambda \in \C \text{ and } &\phi(a) -\lambda 1_B \text{ is not invertible in $B$} \} \label{diplayed-eqn} \tag{$\dagger$} \\ 
& \subseteq 
\{ | \lambda |  : \lambda \in \C \text{ and  } a-\lambda 1_A \text{ is not invertible in $A$} \} \notag
\end{align}
so that $r(\phi(a)) \leq r(a)$, and $\| \phi(a) \| \leq \| a \|$ for any selfadjoint $a \in A$.  For a general $a \in A$, we may use the result of the previous sentence and the $C^*$-identity to obtain $\| \phi(a) \|^2 = \| \phi (a)^* \phi(a) \| = \| \phi(a^* a) \| \leq \| a^* a \|  = \| a \|^2$, so that $\| \phi(a) \| \leq \| a \|$ for all $a \in A$.

When $\phi$ is a unital $*$-preserving homomorphism that is not necessarily linear,  the proof above fails because one cannot conclude that $\phi (a - \lambda 1_A)$ is equal to $\phi(a) - \lambda 1_B$, and thus cannot deduce the containment shown in \eqref{diplayed-eqn}.  Nonetheless, we are still able to prove the result by developing a new approach that uses some elementary $C^*$-algebra techniques.

\section{The Main Result}

\begin{definition}
Let $A$ and $B$ be $C^*$-algebras.  A \emph{ring homomorphism} from $A$ to $B$ is a function $\phi : A \to B$ that is additive (i.e., $\phi(a + b) = \phi(a) + \phi(b)$ for $a,b \in A$) and multiplicative (i.e., $\phi(ab) = \phi(a) \phi(b)$ for $a,b \in A$).  A ring homomorphism is \emph{$*$-preserving} if $\phi(a^*) = \phi(a)^*$ for $a \in A$.  
When $A$ and $B$ are unital with units $1_A \in A$ and $1_B \in B$, respectively, we say $\phi$ is \emph{unital} when $\phi(1_A) = 1_B$.
\end{definition}

\begin{remark}
There are examples of unital $*$-preserving ring homomorphisms that are not linear (and hence not $*$-homomorphisms in the category of $C^*$-algebras).  For example, the map $\phi : \C \to \C$ given by complex conjugation, $\phi(z) = \overline{z}$, is a ring homomorphism that is unital and $*$-preserving, but not linear.  
\end{remark}
The proof of our main result will involve a common trick of relating the value of the norm to positivity of certain $2 \times 2$ matrices.  Recall that an element $a$ of a $C^*$-algebra $A$ is \emph{positive} if $a$ is selfadjoint and the spectrum of $a$ is a subset of the non-negative real numbers.  One can show that $a$ is positive if and only if $a = v^*v$ for some $v \in A$ \cite[Theorem~2.2.5]{Mur}.  In addition, if $A$ is represented as bounded operators on a Hilbert space $\Hi$, then an operator $T$ in $A$ is a positive element if and only if $\langle Tx, x \rangle \geq 0$ for all $x \in \Hi$.

\begin{remark}
If $A$ is a $C^*$-algebra, then $M_2(A)$ denotes the collection of $2 \times 2$ matrices with entries in $A$.  If $\phi : A \to B$ is a function we define $\phi_2 : M_2(A) \to M_2(B)$ by $\phi_2 \left( \begin{smallmatrix} a & b \\ c & d \end{smallmatrix} \right) = \left( \begin{smallmatrix} \phi(a) &  \phi(b)  \\  \phi(c)  &  \phi(d)  \end{smallmatrix} \right)$.  If $A$ is a $C^*$-algebra, then $M_2(A)$ is a $C^*$-algebra. If $A$ is represented on Hilbert space via a representation $\pi : A \to B (\Hi)$, then $M_2(A)$ is represented via $\pi_2 : M_2(A) \to M_2(B(\Hi)) \cong B(\Hi \oplus \Hi)$.  Recall that the Hilbert space direct sum $\Hi \oplus \Hi$ has inner product $\left\langle  \left( \begin{smallmatrix} x \\ y \end{smallmatrix} \right),  \left( \begin{smallmatrix} z \\ w \end{smallmatrix} \right) \right\rangle = \langle x, z \rangle + \langle y, w \rangle$.  Furthermore, note that if $\phi : A \to B$ is a unital $*$-preserving ring homomorphism, then $\phi_2 : M_2(A) \to M_2(B)$ is a unital $*$-preserving ring homomorphism.  These notions and the following lemma are found in \cite[Ch.~1--3]{Pau} and \cite[Lemma~3.1]{Pau}.
\end{remark}

\begin{lemma} \label{norm-pos-lem}
Let $A$ be a unital $C^*$-algebra and suppose that $a \in A$.  Then $\| a \| \leq r$ if and only if $$ \begin{pmatrix} r 1_A & a \\ a^* & r 1_A \end{pmatrix} \geq 0$$ in $M_2(A)$.
\end{lemma}

\begin{proof}
We represent $A$ on Hilbert space $\pi : A \to B(\Hi)$ and set $T := \pi(a)$.  If $\| T \| \leq r$, then for any vectors $x, y \in \Hi$ we have 
\begin{align*}
\left\langle \begin{pmatrix} r I & T \\ T^* & r I \end{pmatrix}  \begin{pmatrix} x \\ y \end{pmatrix},  \begin{pmatrix} x \\ y \end{pmatrix} \right\rangle &= r \langle x, x \rangle + \langle Ty, x \rangle + \langle T^* x, y \rangle + r \langle y, y \rangle \\
&\geq r \| x \|^2 - 2 \| T \| \| x \| \| y \| + r \| y \|^2 \geq 0.
\end{align*}
Conversely, if $\| T \| > r$, then there exist unit vectors $x, y \in \Hi$ such that $\langle Tx, y \rangle < -r$ and the above inner product is negative.
\end{proof}

\begin{lemma} \label{Q-lin-lem}
If $\phi : A \to B$ is a unital ring homomorphism between unital $C^*$-algebras, then $\phi (k a) = k \phi(a)$ for all $k \in \mathbb{Q}$.
\end{lemma}

\begin{proof}
Since $\phi$ is additive, $\phi(na) = n \phi(a)$ for all $n \in \mathbb{Z}$.  In addition, since $\phi$ is a unital ring homomorphism, $\phi(a^{-1}) = \phi(a)^{-1}$ whenever $a \in A$ and  $a$ is invertible.  For any $m \in \mathbb{Z}$, we see that $m1_A$ is invertible with inverse $(1/m)1_A$, and $\phi((1/m)1_A) = \phi( (m1_A)^{-1}) = \phi(m1_A)^{-1} = (m \phi(1_A))^{-1} = (m1_B)^{-1} = (1/m) 1_B = (1/m) \phi(1_A)$.  Thus for any $k \in \mathbb{Q}$, we may write $k = n/m$ with $m,n \in \mathbb{Z}$, and then $\phi(ka) = \phi((n/m)1_A) \phi(a) = n\phi((1/m) 1_A) \phi(a) = n (1/m) \phi(1_A) \phi(a) = k \phi(a)$.
\end{proof}

\begin{theorem} \label{main-thm}
If $A$ and $B$ are unital $C^*$-algebras and $\phi : A \to B$ is a unital $*$-preserving ring homomorphism, then $\phi$ is contractive; i.e., $\| \phi (a) \| \leq \| a \|$ for all $a \in A$.  Consequently, $\phi$ is also continuous.
\end{theorem}

\begin{proof}
Suppose $\| a \| = r$.  Let $k \in \mathbb{Q}$ with $k \geq r$.  Then $\| a \| \leq k$, and Lemma~\ref{norm-pos-lem} implies that $\left( \begin{smallmatrix} k 1_A & a \\ a^* & k1_A \end{smallmatrix} \right) \geq 0$ and $\left( \begin{smallmatrix} k 1_A & a \\ a^* & k1_A \end{smallmatrix} \right) = V^* V $ for some $V \in M_2(A)$.  Since $\phi : A \to B$ is a unital $*$-preserving ring homomorphism, we see that $\phi_2 : M_2(A) \to M_2(B)$ is a unital $*$-preserving ring homomorphism.  Thus $\phi_2 \left( \begin{smallmatrix} k 1_A & a \\ a^* & k1_A \end{smallmatrix} \right) = \phi_2 (V^* V ) = \phi_2(V)^* \phi_2(V)$, so $\phi_2 \left( \begin{smallmatrix} k 1_A & a \\ a^* & k1_A \end{smallmatrix} \right) \geq 0$ in $M_2(B)$.  However, using the definition of $\phi_2$ and Lemma~\ref{Q-lin-lem} we have $$\phi_2 \left( \begin{smallmatrix} k 1_A & a \\ a^* & k1_A \end{smallmatrix} \right) =  \left( \begin{smallmatrix} \phi(k 1_A) & \phi(a) \\ \phi(a^*) & \phi(k1_A) \end{smallmatrix} \right) =  \left( \begin{smallmatrix} k\phi(1_A) & \phi(a) \\ \phi(a)^* & k\phi(1_A) \end{smallmatrix} \right) = \left( \begin{smallmatrix} k1_B & \phi(a) \\ \phi(a)^* & k 1_B \end{smallmatrix} \right)$$ so that $\left( \begin{smallmatrix} k1_B & \phi(a) \\ \phi(a)^* & k 1_B \end{smallmatrix} \right) \geq 0$ in $M_2(B)$.  It follows from Lemma~\ref{norm-pos-lem} that $\| \phi(a) \| \leq k$.  Hence we have shown that $\| \phi(a) \| \leq k$ for all $k \in \mathbb{Q}$ with $k \geq r$.  Thus $\| \phi(a) \| \leq r$.

The continuity of $\phi$ follows from the fact that any bounded map between metric spaces is continuous. 
\end{proof}

\begin{corollary} \label{phi-R-lin-cor}
If $A$ and $B$ are unital $C^*$-algebras and $\phi : A \to B$ is a unital $*$-preserving ring homomorphism, then $\phi$ is $\mathbb{R}$-linear; i.e., $\phi(ra) = r\phi(a)$ for all $r \in \mathbb{R}$ and $a \in A$.
\end{corollary}

\begin{proof}
Let $a \in A$ and $r \in \mathbb{R}$.  Choose a sequence $\{k_n \}_{n=1}^\infty \subseteq \mathbb{Q}$ with $\lim k_n = r$.  Then $\lim k_n a = ra$, and using Lemma~\ref{Q-lin-lem} and Theorem~\ref{main-thm} we have that $ \phi (ra) = \phi( \lim k_n a) = \lim \phi( k_n a) = \lim k_n \phi(a) = r \phi(a)$.
\end{proof}

\begin{corollary}
Let $A$ and $B$ be unital $C^*$-algebras and let $\phi : A \to B$ be a unital $*$-preserving ring homomorphism.  Then $\phi$ is injective if and only if $\phi$ is isometric (i.e., $\| \phi (a) \| = \| a \|$ for all $a \in A$).  
\end{corollary}

\begin{proof}
If $\phi$ is isometric, then it is easy to see that $\phi$ must be injective, since $\phi(a) = 0$ implies $\| a \| = \| \phi(a) \| = 0$ and hence $a = 0$.  Conversely, suppose that $\phi$ is injective.  Define $\| \cdot \|_0 : A \to [0, \infty)$ by $\| a \|_0 := \| \phi (a) \|$.  We see that $\| \cdot \|_0$ is a seminorm because $\| a + b \|_0 := \| \phi(a+b) \| = \| \phi(a) + \phi(b) \| \leq \| \phi(a) \| + \| \phi(b) \| = \| a \|_0 + \| b \|_0$, and using Corollary~\ref{phi-R-lin-cor}, for any $a \in A$ and $\lambda \in \C$, we have 
\begin{align*}
\| \lambda a \|^2_0 &= \| \phi(\lambda a) \|^2 = \| \phi(\lambda a)^* \phi(\lambda a) \| = \| \phi (\overline{\lambda} a^*) \phi(\lambda a) \| = \| \phi (| \lambda|^2 a^* a) \| \\
&= \| | \lambda|^2  \phi (a^* a) \| = | \lambda |^2 \| \phi(a)^* \phi(a) \| = | \lambda |^2 \| \phi(a) \|^2 = | \lambda |^2 \| a \|_0^2
\end{align*}
so that $\| \lambda a \|_0 = | \lambda | \| a \|_0$.  In addition, $\| \cdot \|_0$ is a norm since whenever $\| a \|_0 = 0$ we have $\| \phi(a) \| = 0$, and $\phi(a) = 0$, and the injectivity of $\phi$ implies that $a = 0$.  Finally, $\| \cdot \|_0$ is a $C^*$-norm on $A$ because $\| a \|_0^2 = \| \phi(a) \|^2 = \| \phi(a)^* \phi(a) \| = \| \phi(a^*a) \| = \| a^*a \|_0$.  Since $A$ is a $C^*$-algebra, there is a unique $C^*$-norm on $A$ (see, for example, \cite[Corollary~2.1.2]{Mur}), and hence $\| \cdot \|_0 = \| \cdot \|$.  It follows that $\| \phi(a) \| = \| a \|$ for all $a \in A$.
\end{proof}

The following proposition shows that any unital $*$-preserving ring homomorphism between $C^*$-algebras decomposes as a direct sum of a linear $*$-preserving ring homomorphism and a conjugate-linear $*$-preserving ring homomorphism.  

\begin{proposition} \label{characterize-prop}
If $A$ and $B$ are unital $C^*$-algebras and $\phi : A \to B$ is a unital $*$-preserving ring homomorphism, then there exist projections $P, Q \in B$ such that $P$ and $Q$ commute with $\im \phi$, the projections satisfy $P + Q = 1_B$, the map $\phi_1 := P\phi$ is a linear $*$-preserving ring homomorphism from $A$ to $B$, and the map $\phi_2 := Q \phi$ is a conjugate-linear $*$-preserving ring homomorphism from $A$ to $B$.  Thus $$\phi = \phi_1 + \phi_2$$ and $\phi$ is the direct sum of a linear $*$-preserving ring homomorphism and a conjugate-linear $*$-preserving ring homomorphism.  

Moreover, if the $C^*$-subalgebra of $B$ generated by $\im \phi$ has trivial center (i.e., the center is $\{ \lambda 1_B : \lambda \in \C \}$), then either $\phi_1 = 0$ or $\phi_2 = 0$, and $\phi$ is either linear or conjugate linear.
\end{proposition}

\begin{proof}
Without loss of generality, we may assume that $\im \phi$ generates $B$.  (If not, simply replace $B$ by the $C^*$-subalgebra of $B$ generated by $\im \phi$.)  Let $T := -i \phi(i1_A)$.  Since $\im \phi$ generates $B$, we see that $T$ is in the center of $B$.  Using the properties of $\phi$ we have that $T^2 = \phi(1_A) = 1_B$ and also $T^* = T$.  Let $P := \frac{1}{2} (T + 1_B)$, and $Q := 1_B - P$.  Then $T = P-Q$, and $P$ and $Q$ are orthogonal projections that are in the center of $B$.  Define $\phi_1 : A \to B$ by $\phi_1(a) := P \phi(a)$, and define $\phi_2 : A \to B$ by $\phi_2(a) := Q \phi(a)$.  Then $\phi_1$ and $\phi_2$ are $*$-preserving ring homomorphisms with mutually orthogonal ranges, and since $P + Q =1_B$ it follows that $\phi = \phi_1 + \phi_2$.  In addition, $\phi_1(i1_A) = P \phi(i1_A) = P i T = P i (P-Q) = iP = i \phi_1(1_A)$, and similarly $\phi_2(i1_A) = Q \phi(i1_A) = QiT = Qi(P-Q) = -iQ = -i \phi_2(1_A)$.  Since $\phi_1$ and $\phi_2$ are real linear by Corollary~\ref{phi-R-lin-cor}, it follows that $\phi_1$ is complex linear and $\phi_2$ is conjugate linear.  Finally, we see that if the $C^*$-subalgebra of $B$ generated by $\im \phi$ has trivial center, then because $P$ and $Q$ are orthogonal projection in this center, it must be the case that either $P=0$ or $Q=0$, and hence either $\phi_1 = 0$ or $\phi_2=0$.
\end{proof}

\begin{corollary}
If $A$ is a unital $C^*$-algebra and $\phi : A \to B(\Hi)$ (respectively, $\phi : A \to M_n(\C)$) is a unital $*$-preserving ring homomorphism that is surjective, then $\phi$ is either linear or conjugate-linear.
\end{corollary}

\begin{corollary}
If $\phi : \C \to \C$ is a unital $*$-preserving ring homomorphism, then either $\phi$ is equal to the identity map (i.e., $\phi(z) = z$ for all $z \in \C$) or $\phi$ is equal to the complex conjugation map (i.e., $\phi(z) = \overline{z}$ for all $z \in \C$). 
\end{corollary}

\begin{remark}
Readers familiar with the theory of $C^*$-algebras will note that Proposition~\ref{characterize-prop} is reminiscent of a famous theorem of Kadison: any Jordan $*$-isomorphism between von Neumann algebras is the direct sum of a $*$-isomorphism and a $*$-anti-isomorphism \cite[\S10.5.26]{KadRing}.  Note, however, that Proposition~\ref{characterize-prop} is different from Kadison's result.  Specifically, the Jordan $*$-isomorphism, $*$-isomorphism, and $*$-anti-isomorphism of Kadison's result are all assumed to be linear.  In addition, Proposition~\ref{characterize-prop} does not require the homomorphism to be an isomorphism or the $C^*$-algebra to be a von Neumann algebra.
\end{remark}

\begin{remark}
If $A$ and $B$ are unital $C^*$-algebras, then the kernel of a unital $*$-preserving ring homomorphism $\phi : A \to B$ is always a $C^*$-algebra ideal of $A$:  To see that $\ker \phi$ is closed under scalar multiplication note that if $a \in \ker \phi$ and $\lambda \in \C$, then $\phi(\lambda a) = \phi (\lambda 1_A a) = \phi( \lambda 1_A) \phi (a) = \phi( \lambda 1_A) \cdot 0 = 0$, so $\lambda a \in \ker \phi$.  Thus the quotient $A /  \ker \phi$ is a $C^*$-algebra.  However, the image of $\phi$ is not necessarily a $C^*$-algebra.  For example, the map $\phi : \C \to M_2(\C)$ given by $\phi(z) = \left( \begin{smallmatrix} z & 0 \\ 0 & \overline{z} \end{smallmatrix} \right)$ has an image that is not closed under scalar multiplication, and thus not a $C^*$-algebra.
\end{remark}

\begin{remark}
Suppose $\phi : A \to B$ is a $*$-preserving ring homomorphism between $C^*$-algebras that is not necessarily unital.   If $A$ is a unital $C^*$-algebra, and if we let $B_0$ be the $C^*$-subalgebra of $B$ generated by $\im \phi$, then the map $\phi : A \to B_0$ obtained by restricting the domain is a unital $*$-preserving ring homomorphism between $C^*$-algebras.  Thus, while all results in this paper are stated for unital $*$-preserving ring homomorphisms between $C^*$-algebras, it is possible (by restricting the codomain) to apply the results whenever one has a $*$-preserving ring homomorphism from a unital $C^*$-algebra into another $C^*$-algebra.
\end{remark}

\noindent \textbf{Acknowledgements:}  The author thanks the referee for a careful reading and insightful comments.  The author also thanks Vern I.~Paulsen for a conversation leading to the proof of Proposition~\ref{characterize-prop}.  Furthermore, the author is grateful to the referee who provided an improvement to Proposition~3.9 and also pointed out that there is an alternate way to prove Theorem~\ref{main-thm}:  Since $\phi$ is a $*$-preserving ring homomorphism, $\phi (a^* a) = \phi (a)^* \phi (a)$, so $\phi$ preserves positivity and hence preserves order. Fix $\lambda \in \R$.  Using Lemma~\ref{Q-lin-lem}, for every $r, s \in \mathbb{Q}$ with $r \leq \lambda \leq s$, we have $r1_B = \phi (r1_A ) \leq \phi(\lambda 1_A) \leq \phi(s1_A) = s 1_B$ and thus for all $a \in A$ we have $\phi (\lambda a) = \lambda \phi (a)$. Hence $\phi$ is $\R$-linear. Fix $a \in A$.  For any selfadjoint element $x$ of $A$ we have $\| x \| = \min \{ \lambda \in \R : -\lambda 1_A \leq x \leq \lambda 1_A \}$.  Hence for $\lambda := \| a^* a \| = \|a \|^2$ we have $-\lambda 1_A \leq a^*a \leq \lambda 1_A$, and because $\phi$ preserves order, $-\lambda 1_B \leq \phi (a^*a) \leq \lambda 1_B$.  Thus $\| \phi(a) \|^2 = \| \phi(a)^* \phi(a) \| = \| \phi(a^*a) \| \leq \| a^* a \| = \|a \|^2$ and $\| \phi(a) \| \leq \| a \|$.

\end{document}